\documentclass[11pt]{amsart}

\usepackage[utf8]{inputenc}
\usepackage[backend=bibtex,style=numeric]{biblatex}
\usepackage[vcentermath,enableskew]{youngtab}
\usepackage{tikz}
\usepackage[margin=1.3in]{geometry}
\usepackage{hyperref}
\hypersetup{colorlinks}
\usepackage{xurl}
\usepackage{tikz-cd}
\usepackage{comment}
\usepackage{mathrsfs}
\usepackage{float}
\usepackage{lipsum}
\usepackage{booktabs}

\def\SYT{\mathrm{SYT}}
\def\Res{\mathsf{Res}}

\DeclareMathOperator\sh{sh}

\DeclareMathOperator{\Span}{span}

\def\ev{\mathsf{ev}}

\def\C{\mathbb C}
\def\Z{\mathbb Z}

\def\S{\mathfrak{S}}
\def\H{\mathcal{H}}

\def\SYT{\mathrm{SYT}}

\newcommand{\bB}{\mathbb{B}}

\newcommand{\BB}{\mathbb{B}}
\newcommand{\CC}{\mathbb{C}}

\newcommand{\NN}{\mathbb{N}}

\newcommand{\QQ}{\mathbb{Q}}

\newcommand{\ZZ}{\mathbb{Z}}

\renewcommand{\phi}{\varphi}

\newcommand{\ol}[1]{\overline{#1}}

\DeclareMathOperator{\Aut}{Aut}

\DeclareMathOperator{\sig}{sig}
\DeclareMathOperator{\id}{id}

\newcommand\into\hookrightarrow
\newcommand\onto\twoheadrightarrow

\newcommand{\nc}{\newcommand}
\nc{\la}{\lambda}
\nc{\Iso}{\mathsf{Iso}}
\nc{\Irr}{\mathsf{Irr}}
\nc{\Id}{\mathrm{Id}}

\numberwithin{equation}{section}
\newtheorem{Theorem}[equation]{Theorem}
\newtheorem{Proposition}[equation]{Proposition}
\newtheorem{Lemma}[equation]{Lemma}

\newtheorem{Definition}[equation]{Definition}

\newtheorem{Remark}[equation]{Remark}
\theoremstyle{definition}
\newtheorem{Example}[equation]{Example}

\begin{document}

\title{Relations between generalised Gelfand-Tsetlin and Kazhdan-Lusztig bases of $S_n$}

\author{Ali Haidar}
\address{Ali Haidar: School of Mathematics and Statistics, University of Sydney, Australia}
\curraddr{}
\email{\href{mailto:ahai8016@uni.sydney.edu.au}{ahai8016@uni.sydney.edu.au}}
\thanks{}

\author{Oded Yacobi}
\address{Oded Yacobi: School of Mathematics and Statistics, University of Sydney, Australia}
\curraddr{}
\email{\href{mailto:oded.yacobi@sydney.edu.au}{oded.yacobi@sydney.edu.au}}
\thanks{}
\date{}

\begin{abstract} 
We prove that the Kazhdan-Lusztig basis of Specht modules is upper triangular with respect to all generalized Gelfand-Tsetlin bases constructed from any multiplicity-free tower of standard parabolic subgroups. 
\end{abstract}

\maketitle

\section{Introduction}

In their classical paper ``Relations between Young's natural and the Kazhdan-Lusztig representations of $S_n$'' \cite{GM88}, Garsia and MacLarnan show that Young's natural basis of a Specht module is unitriangular with respect to the Kazhdan-Lusztig (KL) basis. The aim of this note is to prove an analogue of this theorem for all generalised Gelfand-Tsetlin (GT) bases.

To explain more precisely, first note that  Young's natural basis is known to be upper-triangular with respect to either Young's seminormal or orthogonal bases \cite{AH21}. These bases are scalar multiples of each other, and are in fact examples of the standard GT basis. 
Recall that the latter is the basis obtained by restricting a Specht module along the tower of groups
\begin{align}\label{eq:tower}
1 < S_2 <S_3<\cdots <S_n.
\end{align}
It is well-defined up to scalar.
By combining the results of \cite{GM88} and \cite{AH21}, one can show that the GT basis is upper triangular with respect to the KL basis. Of course, to make this statement precise one has to be careful about the orderings of the bases, but we ignore that for the moment.

Now observe that, on the one hand, the GT basis depends on a choice of tower \eqref{eq:tower}. Indeed, there are many possible choices, and each one results in a different GT basis which we call ``generalised GT bases'' (Definition \ref{def:genGTbasis}). On the other hand, the KL basis is \emph{canonical} and so is independent of any choices. Therefore, it is reasonable to expect that the KL basis is upper triangular with respect to \emph{any} tower. This is exactly what we prove (Theorem \ref{thm:main}). 

Our proof of Theorem \ref{thm:main} is  independent of Garsia and MacLarnan's work. Moreover, the generalised GT bases are not upper triangular with respect to each other (cf. Example \ref{example}), so it is hopeless to try to bootstrap our result from the standard case to all generalised GT bases in the obvious manner. 

Instead, our argument is based on two important inputs: the classical result about the action of the longest element on the KL basis (Theorem \ref{thm:MBS}), and recent work of the second author with Gossow on the compatibility of the KL basis with restriction (Theorem \ref{prop:GY}). We also note that in order to formulate our theorem we introduce variations of the evacuation operator and dominance order on standard Young tableaux (Definition \ref{def:bij-and-ord}), which may be interesting in their own right. 

\section{Background}

Throughout we work over the field of rational numbers $\QQ$.  
Let $[a,b]=\{a,a+1,\hdots,b\}$, for $1\leq a < b$. Given a set $X$, $S_X=\Aut(X)$ denotes the group of permutations of $X$ and $S_n=S_{[1,n]}$. Let $w_0 \in S_n$ be the longest element, which interchanges $i$ and $n-i+1$ for every $i$, and more generally for an interval $I \subseteq [1,n]$, let $w_I \in S_I$ be the longest element.  Given a subgroup $H \leq S_n$ set $\ol{H}=w_0Hw_0$.

For groups $H < G$ and a representation $V$ of $G$, let $\Res^G_H(V)$ denote the restriction of $V$ to $H$. Given an isomorphism class $\la$ of an irreducible representation of $G$, let $\Iso_\la(V)$ denote the $\la$-isotypic component of $V$.
The irreducible Specht module of $S_n$ indexed by a partition $\la \vdash n$ is denoted $V^\la$. 

Recall that $\dim(V^\la)=\#\SYT(\la)$, where $\SYT(\la)$ denotes the set of standard Young tableaux of shape $\la$. These are $\la$-shaped arrays of numbers which are increasing along rows from left to right, and down columns. For example, ${\scriptsize\young(124,35)} \in \SYT(3,2)$. We number the rows from top to bottom, and let $\varepsilon_k(T)$ denote the row number containing $k$. For example, for the tableau above, $\varepsilon_3(T)=2$ and $\varepsilon_4(T)=1$. Let $\sh(T)$ denote the shape of the tableau $T$.

Recall Sch\"utzenberger's evacuation operator $\ev:\SYT(\la) \to \SYT(\la)$ is an involution on the set of standard Young tableaux \cite[Def. A1.2.8]{StanleyVol2}. For convenience, we also write $\ol{T} := \ev(T)$. Most importantly for us, evacuation encodes the action of $w_0$ on the Kazhdan-Lusztig basis (cf. Theorem \ref{thm:MBS} below). For $\la \vdash n$, set:
\begin{align*}
\SYT_k(\la) &=\{T \in \SYT(\la)\mid \varepsilon_n(T)=k \},\\
\SYT_{\ol{k}}(\la) &=\{T \in \SYT(\la)\mid \varepsilon_n(\ol{T})=k \}.
\end{align*}
Let $d(T)$ denote the tableau obtained from $T \in \SYT(\la)$ by removing the $n$-box. Conversely, suppose $\mu \vdash n-1$ is obtained by removing a box from $\la$. Then for $S \in \SYT(\mu)$,  $S^+ \in \SYT(\la)$ is the tableau obtained by adding an $n$-box in the correct row. The notation $S^+$ is ambiguous, as it depends on $\la$, but this will always be clear from context.

Of central interest for us is the Kazhdan-Lusztig (KL) basis of $V^\la$, denoted $\{c_T \mid T \in \SYT(\la)\}$. This is a canonical basis obtained from the (left) KL cellular representation of the  Hecke algebra of $S_n$ \cite{KL79, GM88}. The KL basis has close connections to the geometry of Schubert varieties and the Kazhdan-Lusztig conjectures. It is difficult to compute in general (it is precisely as difficult to compute as the Kazhdan-Lusztig polynomials), and in particular the action of a generic permutation on basis elements $c_T$ is hard to express.
Nevertheless, we have the following remarkable classical result about the action of the longest element, which is due separately to Berenstein-Zelevinsky, Mathas and Stembridge:

\begin{Theorem}\cite{BZ96, Mathas96, Stem96}\label{thm:MBS}
Let $\la \vdash n$ and let $T \in \SYT(\la)$. Then $w_0\cdot c_T = \pm c_{\ol{T}}$, and the sign depends only on $\la$.
\end{Theorem}

In recent work of the second author with Gossow, we generalised the above theorem to the class of separable permutations \cite{GY22,GY23}. For present purposes, we recall a result from these works that will play a central role.

Let $\la \vdash n$ and suppose $\la$ has $r$ removable boxes, which appear in rows $1\leq a_1 <a_2 \cdots < a_r$. Let $\mu_k \vdash n-1$ be the partition obtained by removing the box in row $a_k$, and let $V_{a}^\la := \Span\{c_T \mid \varepsilon_n(T) \leq a \}$. This defines a filtration:
$$
0 \subset V_{ a_1}^\la \subset V_{ a_2}^\la \subset \cdots \subset V_{ a_r}^\la=V^\la.
$$

\begin{Proposition}\cite[Theorem 1.2]{GY22}\label{prop:GY}
The subspaces $V_{ a_k}^\la$ are $S_{n-1}$ invariant, and the map 
\[
f_k:V_{ a_k}^\la/V_{ a_{k-1}}^\la \longrightarrow V^{\mu_k}, \;\; c_T+V_{ a_{k-1}}^\la \longmapsto \;c_{d(T)}
\]
is an isomorphism of $S_{n-1}$-modules.
\end{Proposition}

Below we use binary sequences to label chains of subgroups. Let $\BB_n$ be the set of binary sequences of length $n$. Given $b=b_1\cdots b_n \in \BB_n$, set 
\begin{align*}
    \ol{b}=\ol{b_1}\cdots\ol{b_n},\\
    b_\dagger = b_2\cdots b_n,
\end{align*}
where $\ol{0}=1$ and vice-versa. 
\section{Generalised Gelfand-Tsetlin bases}\label{sect:gt-bases}


A \textbf{multiplicity-free chain} of $S_n$ is a sequence of groups $(G_1,G_2,\hdots, G_{n-1})$ such that 
\begin{align*}
    S_n=G_1 \supset G_{2} \supset \cdots \supset G_{n-1},
\end{align*}
and $G_m$ is a standard parabolic subgroup of $S_n$ isomorphic to  $S_{n-m+1}$.  The \textbf{standard chain} is $$(S_n,S_{n-1},\hdots,S_2).$$ Writing $G_m=S_{I_m}$ for $I_m$ an interval, $I_{m+1}$ is obtained from $I_m$ by either removing the smallest or largest element. We can thus index the multiplicity-free chains by binary sequences $b=b_1\cdots b_{n-2}$, where $b_{m+1}=0$ if $I_{m+1}$ is obtained from $I_m$ by  removing the  largest element, and is $1$ otherwise. The standard chain correspond to the zero sequence.

For example, if $n=4$ we have four multiplicity-free chains, corresponding to the four two-step paths in the directed graph below:
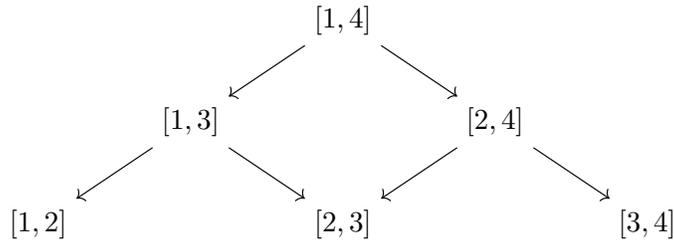
\[\begin{tikzcd}
 &  & \left[1,4\right] \ar[dl]\ar[dr] &  &  \\
 & \left[1,3\right] \ar[dl]\ar[dr] &  & \left[2,4\right] \ar[dl]\ar[dr] &  \\
\left[1,2\right] & & \left[2,3\right] & & \left[3,4\right]
\end{tikzcd}\]

\noindent From left to right, these chains are labelled by $00,01,10$ and $11$. In general, $S_n$ has $N:=2^{n-2}$ multiplicity-free chains. Note that our multiplicity-free chains stop at rank $S_2$ since there is no ``choice'' of $S_1$ at the next step.

Let $(G_1,\hdots,G_{n-1})$ be a multiplicity-free chain labelled by $b$, and let $V$ be an irreducible representation of $S_n$.
\begin{Definition}\label{def:genGTbasis}
     A vector $v \in V$ is a \textbf{$b$-Gelfand-Tsetlin (GT) vector} if for $1\leq m <n$, there exist $\mu^m \vdash n-m+1$ such that
    $$
    v \in \Iso_{\mu^m}\Res^{S_n}_{G_m}(V).
    $$
    In this case, we say $v$ is of \textbf{type $T$}, where $T \in \SYT(\lambda)$ corresponds to the sequence of partitions $(\mu^1,\mu^2,\hdots,\mu^{n-1},(1))$. A basis $\{v_T^b \mid T\in \SYT(\la) \}$ is a \textbf{$b$-GT basis} if for all $T$, $v_T^b$ is a $b$-GT vector of type $T$. We call these \textbf{generalised GT bases}.
\end{Definition}

The following proposition is an immediate consequence of the classical fact that the restriction of any irreducible representation of $G_m$ to $G_{m+1}$ is multiplicity-free:

\begin{Proposition}
    Let $\la \vdash n$ and let $b \in \bB_{n-2}$. Then there exists a $b$-GT basis $\{v_T^b \mid T\in \SYT(\la) \}$ of $V^\la$ which is unique up to scalar.
\end{Proposition}

The $b$-GT basis is related to the standard GT basis by the action of $S_n$, as we now explain. 
First note that a permutation $w\in S_n$ acts  on a multiplicity-free chain by conjugation: 
\[
w\cdot (G_1,\hdots,G_{n-1}) = (wG_1w^{-1},\hdots,wG_{n-1}w^{-1}).
\]
This may not be  a multiplicity-free chain since the conjugate of a standard parabolic subgroup is not necessarily a standard parabolic. But the multiplicity-free chains are always related in this way. This is best explained by example.


Consider the  two chains in $S_6$ in the figure below.  To simplify the picture we replaced the intervals with nodes, so here the red path is the standard chain and the blue path is the chain $([1,6] , [2,6] , [2,5] , [2,4] , [3,4])$. We'll find a permutation $u \in S_6$ that conjugates the standard chain to the blue one.

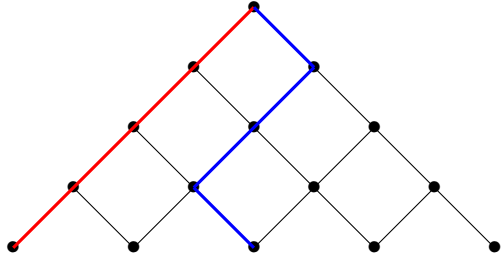
\usetikzlibrary{shapes.geometric}
\[
\begin{tikzpicture}[scale = 0.8]
  \def\height{5}
  \def\width{2*\height-1}

  \foreach \row in {1,...,\height} {
    \foreach \col in {1,...,\row} {
      \pgfmathtruncatemacro\xcoord{\width - \row + 2*\col}
      \pgfmathtruncatemacro\ycoord{\height - \row}
      \node at (\xcoord, \ycoord) {$\bullet$};
    }
  }

  \foreach \row in {1,...,\height} {
    \foreach \col in {1,...,\row} {
      \pgfmathtruncatemacro\xcoord{\width - \row + 2*\col}
      \pgfmathtruncatemacro\ycoord{\height - \row}
      \ifnum\row<\height 
        \draw (\xcoord, \ycoord) -- (\xcoord-1, \ycoord-1);
        \draw (\xcoord, \ycoord) -- (\xcoord+1, \ycoord-1);
      \fi
    }
  }

    \draw[red, very thick] (10,4) -- (6,0);

    \draw[blue, very thick] (10,4) -- (11,3);

    \draw[blue, very thick] (11,3) -- (9,1);

    \draw[blue, very thick] (9,1) -- (10,0);
  
\end{tikzpicture}
\]

\noindent The key point is that given a multiplicity-free chain, suppose the $j$-th group is the symmetric group on an interval $I$. Then the action of the long element $w_{I}$ reflects the chain about the vertical ray pointed downwards from the node labelled by $I$. It doesn't affect the part of the chain above that node. For example, the action of the long element $w_{[2,6]}$ flips the blue chain above to:

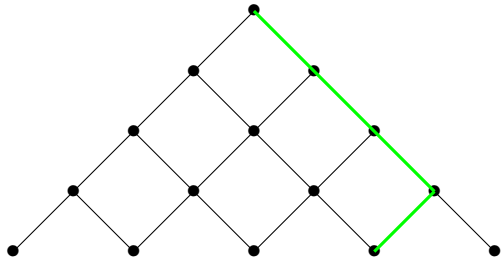
\usetikzlibrary{shapes.geometric}
\[
\begin{tikzpicture}[scale = 0.8]
  \def\height{5}
  \def\width{2*\height-1}

  \foreach \row in {1,...,\height} {
    \foreach \col in {1,...,\row} {
      \pgfmathtruncatemacro\xcoord{\width - \row + 2*\col}
      \pgfmathtruncatemacro\ycoord{\height - \row}
      \node at (\xcoord, \ycoord) {$\bullet$};
    }
  }

  \foreach \row in {1,...,\height} {
    \foreach \col in {1,...,\row} {
      \pgfmathtruncatemacro\xcoord{\width - \row + 2*\col}
      \pgfmathtruncatemacro\ycoord{\height - \row}
      \ifnum\row<\height 
        \draw (\xcoord, \ycoord) -- (\xcoord-1, \ycoord-1);
        \draw (\xcoord, \ycoord) -- (\xcoord+1, \ycoord-1);
      \fi
    }
  }


    \draw[green, very thick] (10,4) -- (13,1);

    \draw[green, very thick] (13,1) -- (12,0);

  
\end{tikzpicture}
\]

Now, if we begin with the red path and  reflect at the first node, then at the second node and then at the fourth node we obtain the blue path.  These nodes correspond to the intervals $[1,6], [2,6]$ and $[2,4]$ respectively. Therefore the desired permutation is given by $u=w_{[2,4]}w_{[2,6]}w_{[1,6]} \in S_6^{sep}$, and is uniquely determined by the chain we started with. We note that $u$ is always a separable permutation, but we won't use this here.


\begin{Lemma}\label{lem:GTcompat}
Let $(G_1,\hdots,G_{n-1})$ be a multiplicity-free chain labelled by $b$ and corresponding to the permutation $u$.
Then for every $T \in \SYT(\la)$, $u\cdot v_T^0=v_T^b$ (up to scalar). Consequently, $w_0\cdot v_T^b = v_T^{\ol{b}}$.
\end{Lemma}

 \begin{proof}
Let $(\mu^1,\hdots,\mu^n)$ be the sequence of partitions corresponding to $T$.
The vector $v_T^0$ is determined (up to scalar) by the containments $v_T^0 \in \Iso_{\mu^m}(\Res_{S_{n-m+1}}^{S_n}(V^\la))$. Since $uS_{n-m+1}u^{-1} = G_m$, $u\cdot v_T^0 \in \Iso_{\mu^m}(\Res_{G_m}^{S_n}(V^\la))$, and so $u\cdot v_T^0=v_T^b$ (up to scalar).

Now, note that $\ol{b}$ labels the multiplicity-free chain $(\ol{G_1},\hdots,\ol{G}_{n-1})$, and the corresponding permutation for this chain is $w_0u$. Therefore $w_0\cdot v_T^b = w_0u\cdot v_T^{0}=v_T^{\ol{b}}$. 
 \end{proof}

\section{The main theorem}
To state our main result we introduce a family of partial orders and bijections on $\SYT(\la)$ labelled by binary sequences.

\begin{Definition}\label{def:bij-and-ord}
    Let $\la \vdash n$ and $b \in \BB_{n-2}$.
    \begin{enumerate}
\item Define a bijection $\phi_b=\phi_{\la,b} \in \Aut(\SYT(\la))$ recursively:
\begin{align*}
    \phi_{\la,b}(T) = \begin{cases}
        \phi_{\mu,b_\dagger}(d(T))^+ &\text{ if } b_1=0 ,\\
         \phi_{\mu,\ol{b_\dagger}}(d(\ol{T}))^+ &\text{ if } b_1=1,
    \end{cases}
\end{align*}
where $\mu=\sh(d(T))$ or $\sh(d(\ol{T}))$.
\item Define a partial order $\leq_b=\leq_{\la,b}$ recursively: $S \leq_{\la,b} T$ if and only if:
\begin{align*}
\begin{cases}
\varepsilon_n(S) < \varepsilon_n(T) \text{ or } \varepsilon_n(S) = \varepsilon_n(T) \text{ and } d(S) \leq_{\mu,b_\dagger} d(T) \text{ if } b_1=0, \\
\varepsilon_n(\ol{S}) < \varepsilon_n(\ol{T}) \text{ or } \varepsilon_n(\ol{S}) = \varepsilon_n(\ol{T}) \text{ and } d(\ol{S}) \leq_{\mu,\ol{b_\dagger}} d(\ol{T}) \text{ if } b_1=1,
\end{cases}
\end{align*}
where $\mu=\sh(d(T))$ or $\sh(d(\ol{T}))$.
\end{enumerate}
\end{Definition}

In the table below we describe the bijection and  partial order  in some examples. 
The partial order is described by associating a sequence of numbers $\sig_b(T)$ to a tableau $T$, and ordering the sequences \textit{reverse lexicographically}.

\renewcommand{\arraystretch}{1.5}
\[
  \begin{array}{c|c|c}
     \mathbf{b} &  \mathbf{\phi_b(T)} & \mathbf{\sig_b(T)}  \\
 \hline
00\cdots0 & \id &  \big(\varepsilon_1(T),\hdots,\varepsilon_n(T)\big) \\
\hline
 11\cdots1 & \ol{T} & \big(\varepsilon_1(\ol{T}),\hdots,\varepsilon_n(\ol{T})\big) \\
 \hline
01\cdots1 & (\ol{d(T)})^+ & \Big(\varepsilon_1(\ol{T}),\hdots,\varepsilon_{n-1}(\ol{T}),\varepsilon_n(T)\Big) \\
\hline
10\cdots0 & (\ol{d(\ol{T})})^+ & \Big(\varepsilon_1\big(\ol{d(\ol{T})}\big),\hdots,\varepsilon_{n-1}\big(\ol{d(\ol{T})}\big),\varepsilon_n\big(\ol{T}\big)\Big)
  \end{array}
\]
\vspace{.5cm}

We can now state our main result, and make a few remarks immediately following:

\begin{Theorem}\label{thm:main}
\begin{enumerate}
\item Let $\la \vdash n$ and $b \in \bB_{n-2}$. Then for any $T \in \SYT(\la)$,
    \begin{align}\label{eq:Th}
c_T = a_Tv_{\phi_b(T)}^b + \sum_{S <_b T}a_Sv_{\phi_b(S)}^b,
    \end{align}
    where $a_T\neq 0$. 
\item The $b$-GT basis can be normalised so that the coefficients $a_T,a_S \in \ZZ$.
\end{enumerate}
\end{Theorem}
\begin{Remark}
The coefficients $a_T,a_S$ that appear in \eqref{eq:Th} depend on $b$. This is perhaps not clear from our notation, which we chose to declutter the formula. Note, the set of $b$-GT bases can be normalised such that the coefficients $a_T,a_S$ are minimally integral in the following sense: $\{a_S \mid S \leq_{b} T\}$ is integral and at least one pair of elements are coprime. In this case, the leading term coefficient $a_T$ is independent (up to sign) of $b \in \BB_{n-2}$. We are not aware of an interesting (geometric) interpretation of these coefficients.
\end{Remark}
\begin{Remark}
If $b=0\cdots0$, i.e. we consider  the standard GT basis, then this result can be deduced from  \cite{GM88, AH21}. Indeed, Garsia and MacLarnan prove that the KL basis is integral and unitriangular with respect to Young's natural basis. The latter is not the standard GT basis, but Armon and Halverson prove that it is upper triangular with respect to the seminormal basis, which is the standard GT basis. Moreover, the relevant ordering are compatible, from which the desired result follows. From this argument we lose integral unitriangularity, and indeed in general we cannot impose that the leading coefficient $a_T=1$ and the other coefficients are integral in \eqref{eq:Th}.
\end{Remark}

\begin{Example}\label{example}
    Consider the first non-trivial case: $\la=(2,1)\vdash 3$. We realise $V^\la$ as the subspace in $\CC^3$ of vectors whose entries sum to zero. Set $S={\scriptsize\young(13,2)}$ and $T={\scriptsize\young(12,3)}$. Then $c_S=(1,-1,0)$ and $c_T=(0,1,-1)$.

    If $b=0$ then $\phi_0=id$ and $S<_0T$.
    The $0$-GT basis is $v_S^0=\frac{1}{2}c_S$ and $v_T^0=\frac{1}{2}(1,1,-2)$ (we choose the normalisation so that the change of basis matrix is integral). The upper-triangularity result is simply 
    \begin{align*}
        c_S &= 2v_S^0, \\
        c_T &= v_T^0-v_S^0
    \end{align*}
    
    On the other hand, if $b=1$ then $\phi_1$ interchanges $S$ and $T$ and $T <_1 S$. The $1$-GT basis is $v_S^1=\frac{1}{2}(0,1,-1)$ and $v_T^1=\frac{1}{2}(2,-1,-1)$. 
    The upper-triangularity in this case is
    \begin{align*}
        c_T &= 2v_{\phi_1(T)}^1, \\
        c_S &= v_{\phi_1(S)}^1-v_{\phi_1(T)}^1
    \end{align*}
    Note here that (up to integral scalar) the Kazhdan-Lusztig basis is the unique basis which is integrally upper triangular with respect to both Gelfand-Tsetlin bases. It's an interesting (and open) question to what extent this holds in general.
\end{Example}

\section{Proof of Theorem \ref{thm:main}}
The remainder of the paper is devoted to the proof of Theorem \ref{thm:main}.
Let $b=b_1\cdots b_{n-2} \in \BB_{n-2}$ label the multiplicity-free chain $(G_1,\hdots,G_{n-1})$.
First note that the second part of the theorem is a trivial consequence of the first: if two bases are related by a rational matrix, then one can rationally rescale one of the bases so that the matrix is integral.
We now proceed to prove Theorem \ref{thm:main}(1)
by induction on $n \geq 2$.

If $n=2$ the result is trivial, 
and we normalise the GT basis in this case to equal the KL basis.
Suppose now that $n>2$ and the result holds for $S_{n-1}$.
Let $\la \vdash n$ and suppose $\la$ has $r$ removable boxes, which appear in rows $1\leq a_1 <a_2 \cdots < a_r$. Let $\mu_k \vdash n-1$ be the partition obtained by removing the box in row $a_k$.
Then, for any $1\leq k \leq r$ and $T \in \SYT(\mu_k)$, by induction we have an equality in $V^{\mu_k}$:
\begin{align}\label{eq:basecase}
c_{d(T)} = a_{d(T)}v_{\phi_{\mu_k,b_\dagger}(d(T))}^{b_\dagger} 
+ 
\sum_{\substack{S \in \SYT(\mu_k) \\ S <_{\mu_k,b_\dagger} d(T)}}a_Sv_{\phi_{\mu_k,b_\dagger}(S)}^{b_\dagger},
\end{align}
with rational coefficients.

Recall from Proposition \ref{prop:GY} we have a filtration
\[
0 \subset V^\la_{a_1} \subset \cdots \subset V^\la_{a_r}=V^\la,
\]
and an isomorphism of $S_{n-1}$-modules $f_k:V^\la_{a_k}/V^\la_{a_{k-1}} \to V^{\mu_k}$  given by 
$c_T+ V^\la_{a_{k-1}} \longmapsto c_{d(T)}$. Therefore, in the unique decomposition $V^\la=V_1\oplus\cdots\oplus V_r$, where $V_r \cong V^{\mu_k}$ as $S_{n-1}$-modules, $V^\la_{a_k}=V_1\oplus\cdots\oplus V_k$. 
The induction now breaks into two cases depending on the value of $b_1$. 
\subsection{Case 1} We first consider the case $b_{1}=0$.

\begin{Lemma}\label{lem:f_k}
    For any $T \in \SYT(\la)$, $f_k(v_T^b+ V^\la_{a_{k-1}})= v_{d(T)}^{b_\dagger}$ (up to nonzero scalar).
\end{Lemma}
\begin{proof}
    Suppose $T \in \SYT_{a_k}(\la)$. Then, by definition,
    $v_T^b \in V_k$, and for $1 < m < n-1$ there exist $\la^m \vdash n-m+1$ such that 
    \[
v_T^b \in \Iso_{\la^m}\Res^{S_{n-1}}_{G_m}(V_k).
    \]
    This shows that $\{v_T^b\mid T\in \SYT_{a_k}(\la)\}$ is a $b_\dagger$-GT basis of $V_k$, and moreover $v_T^b$ is of type $d(T)$. Since $f_k$ is an isomorphism of $S_{n-1}$ modules, the result follows.
\end{proof}

We normalize the $b$-GT basis of $V^\la$ so that Lemma \ref{lem:f_k} is an equality on the nose. 
Now we induct on $1\leq k \leq r$. Let $k=1$.  
Applying $f_1^{-1}$ to \eqref{eq:basecase} we obtain
\[
c_{T} = a_{d(T)}v_{\phi_{\mu_1,b_\dagger}(d(T))^+}^{b} 
+ 
\sum_{\substack{S \in \SYT(\mu_1) \\ S <_{\mu_1,b_\dagger} d(T)}}a_Sv_{\phi_{\mu_1,b_\dagger}(S)^+}^{b}.
\]
By Definition \ref{def:bij-and-ord}, and setting $a_P:=a_{d(P)}$ for all $P\in \SYT_{a_1}(\la)$, the right hand side can be rewritten as
\[
a_{T}v_{\phi_{\la,b}(T)}^{b} \,
+ 
\sum_{\substack{Q \in \SYT_{a_1}(\la) \\ Q <_{\la,b} T}}a_Qv_{\phi_{\la,b}(Q)}^{b}.
\]
This completes the base case of the induction on $k$.

Now let $k>1$. Applying $f_k^{-1}$ to \eqref{eq:basecase}, and using Definition \ref{def:bij-and-ord} as in the base case, we obtain a congruence:
\[
c_{T} \;\equiv\; a_{T}v_{\phi_{\la,b}(T)}^{b} 
+ 
\sum_{\substack{Q \in \SYT_{a_k}(\la) \\ Q <_{\la,b} T}}a_Qv_{\phi_{\la,b}(Q)}^{b} \;\;\Big(\text{modulo } V_{a_{k-1}}^\la\Big).
\]
Therefore, there exists some $x\in V_{a_{k-1}}^\la$ such that 
\[
c_{T} = a_{T}v_{\phi_{\la,b}(T)}^{b} 
+ 
\sum_{\substack{Q \in \SYT_{a_k}(\la) \\ Q <_{\la,b} T}}a_Qv_{\phi_{\la,b}(Q)}^{b} + x.
\]
Now, $x = \sum_{R}a_Rc_R$, the sum ranging over $R \in \SYT(\la)$ such that $\varepsilon_n(R) \leq a_{k-1}.$
By induction (on $k$), Theorem \ref{thm:main} holds for $c_R$. Since $R <_{\la,b} T$ this completes the induction in this case.

\subsection{Case 2} We now consider the case  $b_{1}=1$.
The proof here follows a similar strategy, but requires some subtle changes. For any $k \in \NN$ define,
\begin{align*}
V_{\ol{k}}^\la &= \Span\{c_T \mid T \in \SYT(\la),\; \varepsilon_n(\ol{T})\leq k \}, \\
f_{\ol{k}}&:V_{\ol{a}_k}^\la/V_{\ol{a}_{k-1}}^\la \to V^{\mu_k}, \; c_T+V_{\ol{a}_{k-1}}^\la \longmapsto \pm c_{d(\ol{T})}.
\end{align*}
where the sign in the definition of $f_{\ol{k}}$ agrees with the one appearing in $w_0\cdot c_T = \pm c_{\ol{T}}$ (cf. Theorem \ref{thm:MBS}), and in particular it only depends on $\la$. 
First note:
\begin{Lemma}
    For any $a$, $w_0(V_a^\la)=V_{\ol{a}}^\la$, and hence $V_{\ol{a}}^\la$ is $\ol{S_{n-1}}$-invariant. 
\end{Lemma}
\begin{proof}
    The first statement is an immediate consequence of that fact that evacuation is an involution and Theorem \ref{thm:MBS}. The second now follows since $V_a^\la$ is $S_{n-1}$-invariant.
\end{proof}
By this lemma we can regard $V_{\ol{a}}^\la$ as and $S_{n-1}$-module by twisting the action with $w_0$, i.e. by transferring the structure via the isomorphism $w_0:V_{a}^\la \to V_{\ol{a}}^\la$.

\begin{Lemma}\label{commdi}
    The map $f_{\ol{k}}$ is an isomorphism of $S_{n-1}$-modules.
\end{Lemma}
\begin{proof}
Consider the diagram
\[
\begin{tikzcd}
    V_{\ol{a}_k}^\la/V_{\ol{a}_{k-1}}^\la \arrow[r, "w_0"] \arrow[dr, "f_{\ol{k}}"'] & V_{a_k}^\la/V_{a_{k-1}}^\la \arrow[d, "f_k"'] \\
    & V^{\mu_k}
\end{tikzcd}
\]
We claim this commutes. Indeed, following the diagram right and down we get:
\begin{align*}
    c_T+V_{\ol{a}_{k-1}}^\la &\mapsto w_0(c_T)+ V_{a_{k-1}}^\la \\
     & = \pm c_{\ol{T}} + V_{a_{k-1}}^\la \\
     &\mapsto \pm c_{d(\ol{T})},
\end{align*}
agreeing with the value of $c_T+V_{\ol{a}_{k-1}}^\la$ under $f_{\ol{k}}$.
Therefore $f_{\ol{k}}$ is a composition of $S_{n-1}$-module isomorphisms.
\end{proof}

\begin{Lemma}\label{lem:f_kbar}
    For any $T \in \SYT(\la)$, $f_{\ol{k}}(v_T^b+ V^\la_{\ol{a}_{k-1}})= \pm v_{d(T)}^{\ol{b}_\dagger}$ (up to nonzero scalar).
\end{Lemma}

\begin{proof}
By Lemma \ref{lem:GTcompat}, $v_T^b=w_0\cdot v_T^{\ol{b}}$. Hence 
\begin{align*}
    f_{\ol{k}}(v_T^b+V_{\ol{a}_{k-1}}^\la) &= f_{\ol{k}}(w_0\cdot v_T^{\ol{b}}+V_{\ol{a}_{k-1}}^\la) \\
    &=  f_{k}(v_T^{\ol{b}}+V_{a_{k-1}}^\la) \\
    &=  v_{d(T)}^{\ol{b}_\dagger}
\end{align*}
Note the second equality uses the commutative diagram in Lemma \ref{commdi} and the third equality uses Lemma \ref{lem:f_k}.
\end{proof}

To continue the proof, now we induct on $1 \leq k \leq r$ in the filtration:
\[
0 \subset V^\la_{\ol{a}_1} \subset \cdots \subset V^\la_{\ol{a}_r}=V^\la.
\]
Let $k=1$.  
Consider \eqref{eq:basecase} with $T$ replaced by $\ol{T}$, and $b$ by $\ol{b}$. Applying $f_{\ol{1}}^{-1}$ to this, we obtain
\[
c_{T} = \pm a_{d(T)}v_{\phi_{\mu_1,\ol{b}_\dagger}(d(\ol{T}))^+}^{b} 
+ 
\sum_{\substack{S \in \SYT(\mu_1) \\ S <_{\mu_1,\ol{b}_\dagger} d(\ol{T})}}a_Sv_{\phi_{\mu_1,\ol{b}_\dagger}(S)^+}^{b}.
\]
As in the previous case, the right hand side can be rewritten as
\[
a_{T}v_{\phi_{\la,b}(T)}^{b} 
+ 
\sum_{\substack{Q \in \SYT_{\ol{a}_1}(\la) \\ Q <_{\la,b} T}}a_Qv_{\phi_{\la,b}(Q)}^{b},
\]
proving the base case of the induction on $k$. One can similarly complete the inductive step as in the previous case, and thus the proof of part (1) of Theorem \ref{thm:main} is done.














\section{acknowledgements}
The authors thank Anne Dranowski and Fern Gossow for helpful conversations. The second author thanks MPIM-Bonn where much of this note was written, and his first teacher in representation theory, Adriano Garsia, for instilling his interest in the combinatorics of Specht modules.
This work is partially supported by the Australian Research Council grant DP230100654.


\printbibliography

\end{document}